\documentclass[12pt]{amsart}
\usepackage{amsmath,amsthm,amssymb,cite}
\DeclareMathOperator{\RE}{Re}

\makeatletter
\@namedef{subjclassname@2010}{%
  \textup{2010} Mathematics Subject Classification}
\makeatother

\allowdisplaybreaks

\usepackage[centertags]{amsmath}

\newtheorem{thm}{Theorem}[section]
\newtheorem{cor}[thm]{Corollary}

\newtheorem{rem}[thm]{Remark}

\newtheorem{defn}[thm]{Definition}

\numberwithin{equation}{section}

\title {Harmonic Univalent Functions Defined by  Post Quantum Calculus Operators}

\author[O. P. Ahuja]{Om P. Ahuja}
\address{Department of Mathematics, Kent State University, Burton, OH, USA }
\email{oahuja@kent.edu}

\author[A. \c{C}etinkaya]{Asena \c{C}etinkaya}
\address{Department of Mathematics and Computer Sciences, \.{I}stanbul K\"{u}lt\"{u}r University, \.{I}stanbul, Turkey }
\email{asnfigen@hotmail.com}

\author{V. Ravichandran}
\address{Department of Mathematics, National Institute of Technology Tiruchirappalli-620015, India }
\email{ravic@nitt.edu; vravi68@gmail.com}


\newcommand{\de}{{\mathbb D}}

\begin{document}
\maketitle

\begin{abstract}
We study a  family of  harmonic univalent functions in the open unit disc defined by using  post quantum calculus operators. We first obtained a  coefficient characterization of these functions. Using this, coefficients estimates, distortion and covering theorems were also obtained. The extreme points of the family  and a radius result were also obtained.  The results obtained include several known results as special cases.
\end{abstract}

\maketitle

\section{Introduction}  Let $\mathcal{A}$ be the class of functions $f$ that are analytic in the open unit disc $\de:=\{z:|z|<1\}$ with the normalization $f(0)=f'(0)-1=0$. A function $f\in\mathcal{A}$ can be expressed in the form
\begin{equation}\label{eq2c}
f(z)=z+\sum\limits_{k=2}^\infty a_{k}{z^k},\quad z\in\de.
\end{equation} The theory of $(p,q)$-calculus  (or post quantum calculus)  operators are used in various areas of science and also in the geometric function theory. Let  $0<q\leq p\leq 1$. The $(p,q)$-bracket or twin-basic number $[k]_{p, q}$ for is defined by
$$[k]_{p,q}=\frac{p^k-q^{k}}{p-q}\quad (q\neq p), \quad \text{and }\quad  [k]_{p,p}= kp^{k-1}.$$ Notice that $\lim_{q\to p} [k]_{p,q}=[k]_{p,p}$.  For $0< q \leq 1$, $q$-bracket $[k]_q$ for $ k=0,1,2,\cdots$ is given by
$$[k]_{q}=[k]_{1,q}=\frac{1-q^{k}}{1-q} \quad (q\neq 1), \quad \text{and }\quad  [k]_{1}=[k]_{1,1}= k.$$
The $(p,q)$-derivative operator $D_{p, q} $ of a function $f\in\mathcal{A}$ is given by
\begin{equation}\label{eq2k}
	D_{p, q}f(z)=1+\sum_{{k=2}}^{\infty}[k]_{p,q}a_{k} z^{k-1}.
\end{equation}
For a function $f\in\mathcal{A}$, it can be easily seen that
\begin{equation}\label{eq2}
	D_{p,q}f(z)=\frac{f(pz)-f(qz) }{(p-q)z}, \quad (p\neq q, z\neq 0),
\end{equation}
 $(D_{p,q}f)(0)=1$ and $(D_{p,p}f)(z)=f'(z)$.
For definitions and properties of $(p,q)$-calculus, one may refer to \cite{Chakrabarti1991}. The  $(1,q)$-derivative operator $D_{1, q}$ is known as the  $q$-derivative operator and is denoted by $D_{q}$; for $z\neq 0$, it satisfies
\begin{equation}\label{eq2x}
	(D_{q}f)(z)=\frac{f(z)-f(qz) }{(1-q)z}.
\end{equation}
For definitions and properties of $q$-derivative operator, one may refer to \cite{Ahuja2019, Jackson1909, Jackson1910, JacksonF1910, Ismail1990}.

For a function  $h$ analytic in $\de$ and an integer $m\geq 0$,  we define  the $(p,q)$-S\u al\u agean   differential operator   $L^m_{p,q}$, using $(p,q)$-derivative operator, by
\[
L^0_{p,q}h(z) =h(z) \quad \text{ and }\quad L^m_{p,q}h(z) =zD_{p,q}(L^{m-1}_{p,q}(h(z)).\]     For analytic function $g(z)=\sum\limits_{k=1}^\infty b_{k}{z^k}$,
we have
\begin{equation}\label{eq}
L^m_{p,q}g(z)=\sum_{{k=1}}^{\infty}[k]^m_{p,q}b_{k} z^{k}.
\end{equation}	
In particular, for  $h\in\mathcal{A}$ with $h(z)=z+\sum_{{k=2}}^{\infty} a_{k} z^{k}$, we have
\begin{align}\label{eq1b}L^m_{p,q}h(z)& =z+\sum_{{k=2}}^{\infty}[k]^m_{p,q}a_{k} z^{k}.\end{align}

 Let  $\mathcal{H}$ be the family of complex-valued harmonic functions $f  =h+\overline{g}$ defined in  $\mathbb{D}$, where $h$ and $g$ has the following power series expansion
\begin{equation}\label{eq1}
	h(z)=z+\sum\limits_{k=2}^\infty a_{k}{z^k} \quad\text{and}\quad g(z)=\sum\limits_{k=1}^\infty b_{k}{z^k}.
\end{equation}
Note that $f=h+\overline{g}$  is sense-preserving in  $\mathbb{D}$ if and only if $h'(z)\neq0$ in $\mathbb{D}$ and the second dilatation $w$ of $f$ satisfies the condition $|g'(z)/h'(z)|<1$ in $\mathbb{D}$. Let $\mathcal{S_H}$ be a subclass of functions $f$ in $\mathcal{H}$ that are sense-preserving and univalent in $\de$.  Clunie and Sheil-Small studied the class $\mathcal{S_H}$ in their remarkable  paper \cite{Clunie1984}.
For a survey or comprehensive study of the theory of harmonic univalent functions, one may refer to the papers \cite{Ahuja2005, Ahuja2014, Duren2004}. We introduce and study a new subclass of harmonic univalent functions by using $(p,q)$-S\u al\u agean  harmonic differential operator $L^m_{p,q}:\mathcal{H}\rightarrow \mathcal{H}$.  For the functions in the newly introduced family, a coefficient characterization is obtained (Theorem~\ref{thm1}). Using this, coefficients estimates (Corollary~\ref{cor2}), distortion (Theorem~\ref{thm3}) and covering (Corollary~\ref{cor4}) theorems were also obtained. The extreme points of the family  (Theorem~\ref{rep})  and a radius result (Theorem~\ref{radi}) were also obtained.  The results obtained include several known results as special cases.

\section{Main Results}

We define the $(p,q)$-S\u al\u agean  harmonic differential operator $L^m_{p,q}$ of a  harmonic function  $f=h+\overline{g}\in\mathcal{H}$ by
\begin{align}\label{eq1a}
L^m_{p,q}f(z)&=L^m_{p,q}h(z)+(-1)^m\overline{L^m_{p,q}g(z)}
=z+\sum_{{k=2}}^{\infty}[k]^m_{p,q}a_{k} z^{k}+(-1)^m\sum_{{k=1}}^{\infty}[k]^m_{p,q}\overline{b_{k} z^{k}}.
\end{align}
This last expression is obtained  by using  (\ref{eq1b}) and (\ref{eq}) and is motivated by S\u al\u agean\cite{Salagean1998}.
Recall that convolution (or the Hadamard product) of two complex-valued harmonic functions
$$f_1(z)=z+\sum\limits_{k=2}^\infty a_{1k}{z^k}+\sum\limits_{k=1}^\infty\overline{ b_{1k}z^k}\quad \text{and}\quad
f_2(z)=z+\sum\limits_{k=2}^\infty a_{2k}{z^k}+\sum\limits_{k=1}^\infty\overline{b_{2k}z^k}$$
is defined by
$$f_1(z)\ast f_2(z)=(f_1\ast f_2)(z)=z+\sum\limits_{k=2}^\infty a_{1k}a_{2k}{z^k}+\sum\limits_{k=1}^\infty\overline{ b_{1k}b_{2k}z^k},\quad z\in\de.$$

We now introduce a family of $(p,q)$-S\u al\u agean  harmonic univalent functions by using convolution and the  $(p,q)$-S\u al\u agean  harmonic differential operator $L^m_{p,q}$.
\begin{defn} Suppose $i, j\in\{0,1\}$. Let the function $\Phi_i, \Psi_j$ given by
\begin{equation}\label{eq3}
\Phi_i(z)=z+\sum\limits_{k=2}^\infty \lambda_{k}{z^k}+(-1)^{i}\sum\limits_{k=1}^\infty \mu_{k}{\overline{z}^k},
\end{equation}
\begin{equation}\label{eq3a}	
\Psi_j(z)=z+\sum\limits_{k=2}^\infty u_{k}{z^k}+(-1)^{j}\sum\limits_{k=1}^\infty v_{k}{\overline{z}^k}
\end{equation}
be harmonic in $\de$ with $\lambda_{k}>u_k\geq 0$ $(k\geq2)$ and  $\mu_{k}>v_k\geq 0$ $(k\geq1)$. For $\alpha \in[0,1)$, $0<q\leq p\leq 1$, $m\in \mathbb{N}$, $n\in \mathbb{N}_0$, $m>n$ and $z\in\de$, let $\mathcal{S}_H(m,n,\Phi_i,\Psi_j,p,q,\alpha)$ denote the family of harmonic functions $f$ in $\mathcal{H}$ that satisfy the condition
\begin{equation}\label{eq4}
\RE\bigg\{\frac{(L^m_{p,q}f\ast \Phi_i)(z)}{(L^n_{p,q}f\ast \Psi_j)(z)}\bigg\}> \alpha,
\end{equation}
where  $L^m_{p,q}$ is defined  by (\ref{eq1a}).
\end{defn}
Using (\ref{eq1a}), (\ref{eq3}) and (\ref{eq3a}), we obtain
\begin{equation}\label{eq4z}
(L^m_{p,q}f\ast \Phi_i)(z)=z+\sum\limits_{k=2}^\infty \lambda_k[k]_{p,q}^ma_{k}{z^k}+(-1)^{m+i}\sum\limits_{k=1}^\infty \mu_k[k]_{p,q}^mb_{k}{\overline{z}^k},
\end{equation}
and
\begin{equation}\label{eq4y}
(L^n_{p,q}f\ast \Psi_j)(z)=z+\sum\limits_{k=2}^\infty u_k[k]_{p,q}^na_{k}{z^k}+(-1)^{n+j}\sum\limits_{k=1}^\infty v_k[k]_{p,q}^nb_{k}{\overline{z}^k}.
\end{equation}
\begin{defn}\label{defn2} Let $\mathcal{T}\mathcal{S}_H(m,n,\Phi_i,\Psi_j,p,q,\alpha)$ be the family of harmonic functions $f_m=h+\overline{g}_m\in \mathcal{T}\mathcal{S}_H(m,n,\Phi_i,\Psi_j,p,q,\alpha)$  such that $h$ and $g_m$ are of the form
\begin{equation}\label{eq4a}
h(z)=z-\sum\limits_{k=2}^\infty |a_{k}|{z^k} \quad\text{and}\quad g_m(z)=(-1)^{m+i-1}\sum\limits_{k=1}^\infty |b_{k}|{z^k},\quad |b_1|<1.
\end{equation}
\end{defn}
The families of $\mathcal{S}_H(m,n,\Phi_i,\Psi_j,p,q,\alpha)$  and  $\mathcal{T}\mathcal{S}_H(m,n,\Phi_i,\Psi_j,p,q,\alpha)$ include a variety of well-known subclasses of harmonic functions  as well as many new ones. For example,
\begin{enumerate}
\item[(1)] $\mathcal{S}_H(m,n,\alpha)\equiv \mathcal{S}_H(m,n,\frac{z}{(1-z)^2}-\frac{\overline{z}}{(1-\overline{z})^2},
    \frac{z}{1-z}+\frac{\overline{z}}{1-\overline{z}},1,1,\alpha),$\\ $\mathcal{T}\mathcal{S}_H(m,n,\alpha)\equiv \mathcal{T}\mathcal{S}_H(m,n,\frac{z}{(1-z)^2}-\frac{\overline{z}}
    {(1-\overline{z})^2},\frac{z}{1-z}+\frac{\overline{z}}{1-\overline{z}}
    ,1,1,\alpha)$, \cite{Yalcın 2005}.
\item[(2)]  $\mathcal{S}^\ast_H(\alpha)\equiv \mathcal{S}_H(1,0,\frac{z}{(1-z)^2}-\frac{\overline{z}}
    {(1-\overline{z})^2},\frac{z}{1-z}+\frac{\overline{z}}{1-\overline{z}},1,1,\alpha), $\\
$ \mathcal{T}\mathcal{S}^\ast_H(\alpha)\equiv \mathcal{T}\mathcal{S}_H(1,0,\frac{z}{(1-z)^2}-\frac{\overline{z}}
{(1-\overline{z})^2},\frac{z}{1-z}+\frac{\overline{z}}{1-\overline{z}},1,1,\alpha)$, \cite{Jahangiri1999}.
\item[(3)]  $\mathcal{K}_H(\alpha)\equiv \mathcal{S}_H(2,1,\frac{z+z^2}{(1-z)^3}+\frac{\overline{z}
    +\overline{z}^2}{(1-\overline{z})^3},\frac{z}{(1-z)^2}
    -\frac{\overline{z}}{(1-\overline{z})^2},1,1,\alpha),$\\
$\mathcal{T}\mathcal{K}_H(\alpha)\equiv\mathcal{T}\mathcal{S}_H(2,1,
\frac{z+z^2}{(1-z)^3}+\frac{\overline{z}+\overline{z}^2}{(1-\overline{z})^3},
\frac{z}{(1-z)^2}-\frac{\overline{z}}{(1-\overline{z})^2},1,1,\alpha) $, \cite{Jahangiri1998}.
\item[(4)]  $\mathcal{S}^\ast_{H_q}(\alpha)\equiv\mathcal{S}_H(1,0,\frac{z}{(1-z)^2}
    -\frac{\overline{z}}{(1-\overline{z})^2},\frac{z}{1-z}
    +\frac{\overline{z}}{1-\overline{z}},1,q,\alpha), $\\
$ \mathcal{T}\mathcal{S}^\ast_{H_q}(\alpha)\equiv\mathcal{T}\mathcal{S}_H(1,0,\frac{z}{(1-z)^2}-\frac{\overline{z}}{(1-\overline{z})^2},\frac{z}{1-z}+\frac{\overline{z}}{1-\overline{z}},1,q,\alpha)$, \cite{Ahuja2018}.
\item[(5)]  $\mathcal{K}_{H_q}(\alpha)\equiv\mathcal{S}_H(2,1,\frac{z+z^2}{(1-z)^3}+\frac{\overline{z}+\overline{z}^2}{(1-\overline{z})^3},\frac{z}{(1-z)^2}-\frac{\overline{z}}{(1-\overline{z})^2},1,q,\alpha),$\\
$\mathcal{T}\mathcal{K}_{H_q}(\alpha)\equiv\mathcal{T}\mathcal{S}_H(2,1,\frac{z+z^2}{(1-z)^3}+\frac{\overline{z}+\overline{z}^2}{(1-\overline{z})^3},\frac{z}{(1-z)^2}-\frac{\overline{z}}{(1-\overline{z})^2},1,q,\alpha) $.

\item[(6)]
$\mathcal{S}_H(\Phi_i,\Psi_j,\alpha)\equiv
\mathcal{S}_H(0,0,\Phi_i,\Psi_j,1,1,\alpha),$\\
$\mathcal{T}\mathcal{S}_H(\Phi_i,\Psi_j,\alpha)\equiv \mathcal{T}\mathcal{S}_H(0,0,\Phi_i,\Psi_j,1,1,\alpha)$, \cite{Ravichandran}.
\end{enumerate}
We first prove coefficient conditions for the functions in $\mathcal{S}_H(m,n,\Phi_i,\Psi_j,p,q,\alpha)$ and $\mathcal{T}\mathcal{S}_H(m,n,\Phi_i,\Psi_j,p,q,\alpha)$.

\begin{thm}\label{thm1} Let $f=h+\overline{g}$ be such that $h$ and $g$ are given by (\ref{eq1}). Also, let
\begin{equation}\label{eq5}
\sum\limits_{k=2}^\infty\frac{\lambda_k[k]^m_{p,q}-\alpha u_k[k]^n_{p,q}}{1-\alpha}|a_{k}|
+\sum\limits_{k=1}^\infty\frac{\mu_k[k]^m_{p,q}-(-1)^{n+j-(m+i)}\alpha v_k[k]^n_{p,q}}{1-\alpha}|b_{k}|\leq 1,
\end{equation}
be a given $(p,q)$-coefficient inequality for $\alpha\in[0,1)$, $0<q\leq p\leq 1$, $m\in \mathbb{N}$, $n\in \mathbb{N}_0$, $m>n$,  $\lambda_{k}>u_k\geq 0$ $(k\geq2)$ and  $\mu_{k}>v_k\geq 0$ $(k\geq1)$. Then a function

(i) $f=h+\overline{g}$ given by (\ref{eq1}) is a sense-preserving harmonic univalent functions in $\de$ and $f\in\mathcal{S}_H(m,n,\Phi_i,\Psi_j,p,q,\alpha)$ if the inequality in (\ref{eq5}) is satisfied.

(ii) $f_m=h+\overline{g}_m$ given by (\ref{eq4a}) is in the $\mathcal{T}\mathcal{S}_H(m,n,\Phi_i,\Psi_j,p,q,\alpha)$ if and only if the inequality in (\ref{eq5}) is satisfied.
\end{thm}
\begin{proof}\textit{(i).} Using the techniques used in  \cite{Ravichandran}, it is a routine step to prove that $f=h+\overline{g}$ given by (\ref{eq1}) is sense-preserving and locally univalent in $\de$.
Using the fact $\RE(w)>\alpha$ if and only if $|1-\alpha+w|\geq |1+\alpha-w|$, it suffices to show that
\begin{equation}\label{eq5a}
\bigg|1-\alpha+\frac{(L^m_{p,q}f\ast\Phi_i)(z)}{(L^n_{p,q}f\ast\Psi_j)(z)} \bigg|-\bigg|1+\alpha-\frac{(L^m_{p,q}f\ast\Phi_i)(z)}{(L^n_{p,q}f\ast\Psi_j)(z)} \bigg|\geq 0
\end{equation}
In view of (\ref{eq4z}) and (\ref{eq4y}), left side of (\ref{eq5a}) yields
\begin{align*}
\big|(L^m_{p,q}f\ast\Phi_i)(z)&+(1-\alpha)(L^n_{p,q}f\ast\Psi_j)(z)\big|-\big|(L^m_{p,q}f\ast\Phi_i)(z)-(1+\alpha)(L^n_{p,q}f\ast\Psi_j)(z)\big|\\
&=\bigg|(2-\alpha)z+\sum\limits_{k=2}^\infty(\lambda_k[k]^m_{p,q}+(1-\alpha)u_k[k]^n_{p,q})a_kz^k\\
&\quad\quad+(-1)^{m+i}\sum\limits_{k=1}^\infty(\mu_k[k]^m_{p,q}+(-1)^{n+j-(m+i)}(1-\alpha)v_k[k]^n_{p,q})b_k\overline{z}^k\bigg|\\
&\quad\quad-\bigg|-\alpha z+\sum\limits_{k=2}^\infty(\lambda_k[k]^m_{p,q}-(1+\alpha)u_k[k]^n_{p,q})a_kz^k\\ &\quad\quad+(-1)^{m+i}\sum\limits_{k=1}^\infty(\mu_k[k]^m_{p,q}-(-1)^{n+j-(m+i)}(1+\alpha)v_k[k]^n_{p,q})b_k\overline{z}^k\bigg|\\
&\geq (2-2\alpha)|z|-2\sum\limits_{k=2}^\infty(\lambda_k[k]^m_{p,q}-\alpha u_k[k]^n_{p,q})|a_k||z|^{k}\\
&\quad\quad-\sum\limits_{k=1}^\infty(\mu_k[k]^m_{p,q}+(-1)^{n+j-(m+i)}(1-\alpha)v_k[k]^n_{p,q})|b|_k|z|^{k}\\
&\quad\quad-\sum\limits_{k=1}^\infty(\mu_k[k]^m_{p,q}-(-1)^{n+j-(m+i)}(1+\alpha)v_k[k]^n_{p,q})|b|_k|z|^{k}\\
& \geq \big(1-\alpha\big)|z|\bigg[1-\sum\limits_{k=2}^\infty\frac{\lambda_k[k]^m_{p,q}-\alpha u_k[k]^n_{p,q}}{1-\alpha}|a_{k}||z|^{k-1}\\
&\quad-\sum\limits_{k=1}^\infty\frac{\mu_k[k]^m_{p,q}-(-1)^{n+j-(m+i)}\alpha v_k[k]^n_{p,q}}{1-\alpha}|b_{k}||z|^{k-1}\bigg]\\
&> \big(1-\alpha\big)|z|\bigg[1-\bigg(\sum\limits_{k=2}^\infty\frac{\lambda_k[k]^m_{p,q}-\alpha u_k[k]^n_{p,q}}{1-\alpha}|a_{k}|
\\
&\quad +\sum\limits_{k=1}^\infty\frac{\mu_k[k]^m_{p,q}-(-1)^{n+j-(m+i)}\alpha v_k[k]^n_{p,q}}{1-\alpha}|b_{k}|\bigg)\bigg].
\end{align*}
This last expression is non-negative because of the condition  given in (\ref{eq5}). This completes the proof of part (i) of theorem.

\textit{(ii).} Since $\mathcal{T}\mathcal{S}_H(m,n,\Phi_i,\Psi_j,p,q,\alpha)\subset \mathcal{S}_H(m,n,\Phi_i,\Psi_j,p,q,\alpha)$, the sufficient part of part (ii) follows from part (i).

In order to prove the necessary part of part (ii), we assume that $f_m\in\mathcal{T}\mathcal{S}_H(m,n,\Phi_i,\Psi_j,p,q,\alpha)$. We notice that
\begin{align*}
& \RE \bigg\{\frac{(L^m_{p,q}f\ast \Phi_i)(z)}{(L^n_{p,q}f\ast \Psi_j)(z)}- \alpha\bigg\}\\
&=
\RE \bigg\{\frac{(1-\alpha)z-\sum\limits_{k=2}^\infty(\lambda_k[k]^m_{p,q}-\alpha u_k[k]^n_{p,q})a_kz^k}{z-\sum\limits_{k=2}^\infty u_k[k]^n_{p,q}a_{k}{z^k}+(-1)^{m+i+n+j-1}\sum\limits_{k=1}^\infty v_k[k]^n_{p,q}b_{k}{\overline{z}^k}}\\
&\quad\quad+\frac{(-1)^{2m+2i-1}\sum\limits_{k=1}^\infty(\mu_k[k]^m_{p,q}-(-1)^{n+j-(m+i)}\alpha v_k[k]^n_{p,q})b_k\overline{z}^k}{z-\sum\limits_{k=2}^\infty u_k[k]^n_{p,q}a_{k}{z^k}+(-1)^{m+i+n+j-1}\sum\limits_{k=1}^\infty v_k[k]^n_{p,q}b_{k}{\overline{z}^k}}\bigg\}\\
&\geq\frac{(1-\alpha)-\sum\limits_{k=2}^\infty(\lambda_k[k]^m_{p,q}-\alpha u_k[k]^n_{p,q})a_kr^{k-1}-\sum\limits_{k=1}^\infty(\mu_k[k]^m_{p,q}-(-1)^{n+j-(m+i)}\alpha v_k[k]^n_{p,q})b_kr^{k-1}}{1-\sum\limits_{k=2}^\infty u_k[k]^n_{p,q}a_{k}{r^{k-1}}-(-1)^{m+i+n+j}\sum\limits_{k=1}^\infty v_k[k]^n_{p,q}b_{k}{r^{k-1}}}\\
&\geq 0,
\end{align*}
by (\ref{eq4}). The above inequality must hold for all $z\in\de.$ In particular, choosing the values of $z$ on the positive real axis and $z\rightarrow 1^-$, we obtain the required condition (\ref{eq5}). This completes the proof of part (ii) of theorem.

The harmonic  mappings
\begin{equation}
f(z)=z+\sum\limits_{k=2}^\infty\frac{1-\alpha}{\lambda_k[k]^m_{p,q}-\alpha u_k[k]^n_{p,q}}x_kz^k+\sum\limits_{k=1}^\infty\frac{1-\alpha}{\mu_k[k]^m_{p,q}-(-1)^{n+j-(m+i)}\alpha v_k[k]^n_{p,q}}y_k\overline{z}^k,
\end{equation}
where $\sum\limits_{k=2}^\infty|x_k|+\sum\limits_{k=1}^\infty|y_k|=1$, show that the coefficient bound given by (\ref{eq5}) is sharp.
\end{proof}

Theorem \ref{thm1} also  yields the following corollary.
\begin{cor}\label{cor2} For  $f_m=h+\overline{g}_m$  given by (\ref{eq4a}), we have
$$|a_k|\leq \frac{1-\alpha}{\lambda_k[k]^m_{p,q}-\alpha u_k[k]^n_{p,q}}, k\geq 2\quad\text{and}\quad |b_k|\leq \frac{1-\alpha}{\mu_k[k]^m_{p,q}-(-1)^{n+j-(m+i)}\alpha v_k[k]^n_{p,q}}, k\geq1.$$
The result is sharp for each $k$.
\end{cor}
Using Theorem \ref{thm1} (part ii), it is easily seen that the class  $\mathcal{T}\mathcal{S}_H(m,n,\Phi_i,\Psi_j,p,q,\alpha)$ is convex and closed with respect to the topology of locally uniform convergence so that the closed convex hulls of $\mathcal{T}\mathcal{S}_H(m,n,\Phi_i,\Psi_j,p,q,\alpha)$
equals itself. The next theorem determines the extreme points of  $\mathcal{T}\mathcal{S}_H(m,n,\Phi_i,\Psi_j,p,q,\alpha)$.
\begin{thm}\label{rep} Let $f_m=h+\overline{g}_m$ be given by (\ref{eq4a}). Then $f_m\in \operatorname{clco} \mathcal{T}\mathcal{S}_H(m,n,\Phi_i,\Psi_j,p,q,\alpha)$ if and only if $f_m(z)=\sum\limits_{k=1}^\infty(x_kh_k(z)+y_kg_{m_k}(z))$, where
	$$h_1(z)=z,\quad h_k(z)=z-\frac{1-\alpha}{\lambda_k[k]^m_{p,q}-\alpha u_k[k]^n_{p,q}}z^k,(k\geq2),$$ $$g_{m_k}(z)=z+(-1)^{m+i-1}\frac{1-\alpha}{\mu_k[k]^m_{p,q}-(-1)^{n+j-(m+i)}\alpha v_k[k]^n_{p,q}}\overline{z}^k, (k\geq1),$$
	and $\sum\limits_{k=1}^\infty(x_k+y_k)=1$ where  $x_k\geq 0$ and $ y_k\geq0$. In particular, the extreme points of $\mathcal{T}\mathcal{S}_H(m,n,\Phi_i,\Psi_j,p,q,\alpha)$ are $\{h_k\}$ and $\{g_{m_k}\}$. 	
\end{thm}
\begin{proof} For a function $f_m$ of the form $f_m(z)=\sum\limits_{k=1}^\infty(x_kh_k(z)+y_kg_{m_k}(z)),$ where $\sum\limits_{k=1}^\infty(x_k+y_k)=1$, we have
	$$f_m(z)=z-\sum\limits_{k=2}^\infty\frac{1-\alpha}{\lambda_k[k]^m_{p,q}-\alpha u_k[k]^n_{p,q}}x_kz^k+\sum\limits_{k=1}^\infty(-1)^{m+i-1}\frac{1-\alpha}{\mu_k[k]^m_{p,q}-(-1)^{n+j-(m+i)}\alpha v_k[k]^n_{p,q}}y_k\overline{z}^k.$$
	Then $f_m \in \operatorname{clco} \mathcal{T}\mathcal{S}_H(m,n,\Phi_i,\Psi_j,p,q,\alpha)$ because
	$$\sum\limits_{k=2}^\infty\frac{\lambda_k[k]^m_{p,q}-\alpha u_k[k]^n_{p,q}}{1-\alpha}\bigg(\frac{1-\alpha}{\lambda_k[k]^m_{p,q}-\alpha u_k[k]^n_{p,q}}x_k\bigg)+$$
	$$\sum\limits_{k=1}^\infty\frac{\mu_k[k]^m_{p,q}-(-1)^{n+j-(m+i)}\alpha v_k[k]^n_{p,q}}{1-\alpha}\bigg(\frac{1-\alpha}{\mu_k[k]^m_{p,q}-(-1)^{n+j-(m+i)}\alpha v_k[k]^n_{p,q}}y_k\bigg)$$
	$$=\sum\limits_{k=2}^\infty x_k+\sum\limits_{k=1}^\infty y_k=1-x_1\leq 1.$$
	
	Conversely, suppose $f_m \in \operatorname{clco}\mathcal{T}\mathcal{S}_H(m,n,\Phi_i,\Psi_j,p,q,\alpha)$. Then
	$$|a_k|\leq \frac{1-\alpha}{\lambda_k[k]^m_{p,q}-\alpha u_k[k]^n_{p,q}}\quad\text{and}\quad |b_k|\leq \frac{1-\alpha}{\mu_k[k]^m_{p,q}-(-1)^{n+j-(m+i)}\alpha v_k[k]^n_{p,q}}.$$
	Set
	$$x_k=\frac{\lambda_k[k]^m_{p,q}-\alpha u_k[k]^n_{p,q}}{1-\alpha}|a_k|\quad\text{and}\quad y_k=\frac{\mu_k[k]^m_{p,q}-(-1)^{n+j-(m+i)}\alpha v_k[k]^n_{p,q}}{1-\alpha}|b_k|.$$
	By Theorem \ref{thm1} (ii), $\sum\limits_{k=2}^\infty x_k+ \sum\limits_{k=1}^\infty y_k\leq 1.$ Therefore we define $x_1=1-\sum\limits_{k=2}^\infty x_k- \sum\limits_{k=1}^\infty y_k\geq 0$. Consequently, we obtain $f_m(z)=\sum\limits_{k=1}^\infty(x_kh_k(z)+y_kg_{m_k}(z))$ as required.
\end{proof}

For functions in the class $\mathcal{T}\mathcal{S}_H(m,n,\Phi_i,\Psi_j,p,q,\alpha)$, the following theorem gives distortion bounds which in turns yields the covering result for this class.

\begin{thm}\label{thm3} Let $f_m\in \mathcal{T}\mathcal{S}_H(m,n,\Phi_i,\Psi_j,p,q,\alpha)$, $\gamma_k=\lambda_k[k]^m_{p,q}-\alpha u_k[k]^n_{p,q}$, $k\geq2$ and $\phi_k= \mu_k[k]^m_{p,q}-(-1)^{n+j-(m+i)}\alpha v_k[k]^n_{p,q} $, $k\geq1$. If $\{\gamma_k\}$ and $\{\phi_k\}$ are non-decreasing sequences, then we have
\begin{equation}\label{eq18}
|f_m(z)|\leq (1+|b_1|)|z|+\frac{1-\alpha}{\beta}\bigg(1-\frac{\mu_1-(-1)^{n+j-(m+i)}\alpha v_1}{\beta}|b_1|\bigg)|z|^2
\end{equation}
and
\begin{equation}\label{eq19}
|f_m(z)|\geq (1-|b_1|)|z|-\frac{1-\alpha}{\beta}\bigg(1-\frac{\mu_1-(-1)^{n+j-(m+i)}\alpha v_1}{\beta}|b_1|\bigg)|z|^2,
\end{equation}
for all $z\in\de$, where $b_1=f_{\overline{z}}(0)$ and
\begin{align*}
\beta&=\min\{\gamma_2, \phi_2\} =\min\{\lambda_2[2]^m_{p,q}-\alpha u_2[2]^n_{p,q}, \mu_2[2]^m_{p,q}-(-1)^{n+j-(m+i)}\alpha v_2[2]^n_{p,q}\}.
\end{align*}
\end{thm}

\begin{proof} Let  $f_m\in\mathcal{T}\mathcal{S}_H(m,n,\Phi_i,\Psi_j,p,q,\alpha)$. Taking the absolute value of $f_m$, we obtain
\begin{align*}
		|f_m(z)|&\leq (1+|b_1|)|z|+\sum\limits_{k=2}^\infty(|a_k|+|b_k|)|z|^k\\
		&\leq(1+|b_1|)|z|+\sum\limits_{k=2}^\infty(|a_k|+|b_k|)|z|^2\\
		&\leq (1+|b_1|)|z|+\frac{1-\alpha}{\beta}\sum\limits_{n=2}^\infty\bigg(\frac{\beta}{1-\alpha}|a_k|+\frac{\beta}{1-\alpha}|b_k|\bigg)|z|^2\\
		&\leq (1+|b_1|)|z|+\frac{1-\alpha}{\beta}\sum\limits_{k=2}^\infty\bigg(\frac{\lambda_k[k]^m_{p,q}-\alpha u_k[k]^n_{p,q}}{1-\alpha}|a_k|\\
		&\quad\quad+\frac{\mu_k[k]^m_{p,q}-(-1)^{n+j-(m+i)}\alpha v_k[k]^n_{p,q}}{1-\alpha}|b_k|\bigg)|z|^2\\
		&\leq (1+|b_1|)|z|+\frac{1-\alpha}{\beta}\bigg(1-\frac{\mu_1-(-1)^{n+j-(m+i)}\alpha v_1}{1-\alpha}|b_1|\bigg)|z|^2
\end{align*}
This proves (\ref{eq18}). The proof of (\ref{eq19}) is omitted as it is  similar to the proof of (\ref{eq18}).
\end{proof}
The following covering result follows from  the inequality (\ref{eq19}).

\begin{cor}\label{cor4}  Under the hypothesis of Theorem \ref{thm3}, we have
	$$\bigg\{w: |w|< \frac{1}{\beta}\big(\beta-1+\alpha+(\mu_1-(-1)^{n+j-(m+i)}\alpha v_1-\beta)|b_1|\big)\bigg\}\subset f(\de).$$
\end{cor}
\begin{thm} \label{radi} If $f_m\in\mathcal{T}\mathcal{S}_H(m,n,\Phi_i,\Psi_j,p,q,\alpha)$, then  $f_m$ is convex in the disc
$$|z|\leq  \min_{k}\bigg\{\frac{1-b_1}{k[1-\frac{\mu_1-(-1)^{n+j-(m+i)}\alpha v_1}{1-\alpha}b_1]}\bigg\}^{\frac{1}{k-1}},\quad k\geq 2.$$
\end{thm}
\begin{proof} Let $f_m\in\mathcal{T}\mathcal{S}_H(m,n,\Phi_i,\Psi_j,p,q,\alpha)$ and let $r$, $0<r<1$, be fixed. Then $r^{-1}f_m(rz)\in\mathcal{T}\mathcal{S}_H(m,n,\Phi_i,p,q,\alpha)$ and we have
\begin{align*}
 & \sum\limits_{k=2}^\infty k^2(|a_k|+|b_k|)
 =\sum\limits_{k=2}^\infty k(|a_k|+|b_k|)kr^{k-1}\\
&\quad\leq \sum\limits_{k=2}^\infty \bigg(\frac{\lambda_k[k]^m_{p,q}-\alpha u_k[k]^n_{p,q}}{1-\alpha}|a_k|+\frac{\mu_k[k]^m_{p,q}-(-1)^{n+j-(m+i)}\alpha v_k[k]^n_{p,q}}{1-\alpha}|b_k|\bigg)kr^{k-1}\\
&\quad\leq \sum\limits_{k=2}^\infty \bigg(1-\frac{\mu_1-(-1)^{n+j-(m+i)}\alpha v_1}{1-\alpha}|b_1|\bigg)kr^{k-1}\\
&\quad\leq 1-b_1
\end{align*}	
provided
\[ kr^{k-1}\leq \frac{1-b_1}{1-\frac{\mu_1-(-1)^{n+j-(m+i)}\alpha v_1}{1-\alpha}b_1}\] which is true if
\[r\leq \min_{k}\bigg\{\frac{1-b_1}{k[1-\frac{\mu_1-(-1)^{n+j-(m+i)}\alpha v_1}{1-\alpha}b_1]}\bigg\}^{\frac{1}{k-1}},\quad k\geq 2.\qedhere\]	
\end{proof}

\begin{rem} Our results naturally includes several results known for those subclasses of harmonic functions  listed after Definition~\ref{defn2}.
\end{rem}

\end{document}